\documentclass[reqno, 12pt]{amsart}
\pdfoutput=1
\makeatletter
\let\origsection=\section \def\section{\@ifstar{\origsection*}{\mysection}}
\def\mysection{\@startsection{section}{1}\z@{.7\linespacing\@plus\linespacing}{.5\linespacing}{\normalfont\scshape\centering\S}}
\makeatother

\usepackage{amsmath,amssymb,amsthm}
\usepackage{mathrsfs}
\usepackage{mathabx}\changenotsign
\usepackage{dsfont}

\usepackage{xcolor}

\usepackage{tikz}
\usetikzlibrary{calc, positioning, shapes, backgrounds}

\usepackage{verbatim}

\usepackage[backref]{hyperref}
\hypersetup{
    colorlinks,
    linkcolor={magenta!60!white},
    citecolor={green!60!black},
    urlcolor={blue!60!black}
}

\usepackage[open,openlevel=2,atend]{bookmark}

\usepackage[abbrev,msc-links,backrefs]{amsrefs}
\usepackage{doi}

\renewcommand{\PrintDOI}[1]{\doi{#1}}

\usepackage[T1]{fontenc}
\usepackage{lmodern}
\usepackage[babel]{microtype}
\usepackage[english]{babel}

\usepackage{geometry}
\numberwithin{equation}{section}
\numberwithin{figure}{section}

\usepackage{enumitem}
\def\rmlabel{\upshape({\itshape \roman*\,})}

\let\polishlcross=\l
\def\l{\ifmmode\ell\else\polishlcross\fi}

\let\setminus=\smallsetminus

\let\sm=\setminus

\makeatletter
\def\moverlay{\mathpalette\mov@rlay}
\def\mov@rlay#1#2{\leavevmode\vtop{   \baselineskip\z@skip \lineskiplimit-\maxdimen
   \ialign{\hfil$\m@th#1##$\hfil\cr#2\crcr}}}
\newcommand{\charfusion}[3][\mathord]{
    #1{\ifx#1\mathop\vphantom{#2}\fi
        \mathpalette\mov@rlay{#2\cr#3}
      }
    \ifx#1\mathop\expandafter\displaylimits\fi}
\makeatother

\newcommand{\dcup}{\charfusion[\mathbin]{\cup}{\cdot}}

\makeatletter
\def\mfy@rlay#1{\leavevmode\vtop{   \baselineskip\z@skip \lineskiplimit-\maxdimen
   \ialign{\hfil$\m@th#1##$\hfil\cr\nrhup\crcr}}}
\newcommand{\vfusion}[1][\mathord]{\mathpalette\mfy@rlay{#1\cr}}
\makeatother

\DeclareFontFamily{U}  {MnSymbolC}{}
\DeclareSymbolFont{MnSyC}         {U}  {MnSymbolC}{m}{n}
\DeclareFontShape{U}{MnSymbolC}{m}{n}{
    <-6>  MnSymbolC5
   <6-7>  MnSymbolC6
   <7-8>  MnSymbolC7
   <8-9>  MnSymbolC8
   <9-10> MnSymbolC9
  <10-12> MnSymbolC10
  <12->   MnSymbolC12}{}
\DeclareMathSymbol{\powerset}{\mathord}{MnSyC}{180}

\let\epsilon=\varepsilon
\let\eps=\epsilon
\let\rho=\varrho
\let\theta=\vartheta

\let\phi=\varphi

\def\EE{{\mathds E}}
\def\NN{{\mathds N}}

\def\PP{{\mathds P}}

\newcommand{\cA}{\mathcal{A}}

\newcommand{\cH}{\mathcal{H}}

\theoremstyle{plain}
\newtheorem{thm}{Theorem}[section]

\newtheorem{clm}[thm]{Claim}
\newtheorem{cor}[thm]{Corollary}
\newtheorem{lemma}[thm]{Lemma}

\theoremstyle{definition}

\newtheorem{dfn}[thm]{Definition}

\newtheoremstyle{dotless}{}{}{\itshape}{}{\bfseries}{}{ }{}
\theoremstyle{dotless}

\usepackage{scalerel}
\newcommand{\vrhup}[1]{\scaleobj{0.6}{\scalerel*{\rightharpoonup}{#1}}}
\newcommand{\nrhup}{\mathord{\scaleobj{0.6}{\scalerel*{\rightharpoonup}{x}}}}
\newcommand{\wrhup}{\scaleobj{0.6}{\scalerel*{\rightharpoonup}{W}}}
\def\vseq#1{\ThisStyle{%
  \mathord{\vbox{\offinterlineskip\ialign{%
    \hfil##\hfil\cr
    $\SavedStyle{}_{\smash{\vrhup#1}}$\cr
    \noalign{\kern-0.7\scriptspace}
    $\SavedStyle#1$\cr}}}}}
\def\seq#1{\ThisStyle{%
  \mathord{\vbox{\offinterlineskip\ialign{%
    \hfil##\hfil\cr
    $\SavedStyle{}_{\smash{\nrhup}}$\cr
    \noalign{\kern-0.5\scriptspace}
    $\SavedStyle#1$\cr}}}}}
\def\wseq#1{\ThisStyle{%
  \mathord{\vbox{\offinterlineskip\ialign{%
    \hfil##\hfil\cr
    $\SavedStyle{}_{\smash{\wrhup#1}}$\cr
    \noalign{\kern-0.7\scriptspace}
    $\SavedStyle#1$\cr}}}}}

\let\vec=\seq

\usepackage{enumitem}

\DeclareMathOperator{\bw}{bw}

\newcommand{\zetaa}{\zeta_{\rm A}}
\newcommand{\rhoa}{\rho_{\rm A}}
\newcommand{\alphaa}{\alpha_{\rm A}}
\newcommand{\rhop}{\rho_{\rm P}}
\newcommand{\zetap}{\zeta_{\rm P}}
\newcommand{\rhoc}{\rho_{\rm C}}
\newcommand{\xic}{\xi_{\rm C}}
\newcommand{\mc}{M_{\rm C}}
\newcommand{\zetac}{\zeta_{\rm C}}

\let\to=\longrightarrow

\begin{document}

\title[Embedding spanning subgraphs in uniformly dense and inseparable graphs]{Embedding spanning subgraphs in uniformly dense and inseparable graphs}

\author[O. Ebsen]{Oliver Ebsen}
\author[G. S. Maesaka]{Giulia S. Maesaka}
\author[Chr. Reiher]{Christian Reiher}
\author[M. Schacht]{Mathias Schacht}
\author[B. Sch\"{u}lke]{Bjarne Sch\"{u}lke}
\address{Fachbereich Mathematik, Universit\"{a}t Hamburg, Hamburg, Germany}
\email{Oliver.Ebsen@uni-hamburg.de}
\email{Giulia.Satiko.Maesaka@uni-hamburg.de}
\email{Christian.Reiher@uni-hamburg.de}
\email{schacht@math.uni-hamburg.de}
\email{Bjarne.Schuelke@uni-hamburg.de}
\thanks{G.S.\,M.\ and M.\,S.\ are supported by the \emph{European Research Council} (Consolidator Grant PEPCo 724903).}
\thanks{B.\,S.\ is supported by the \emph{German-Israeli Foundation for Scientific Research and Development} (Collaboration Grant No.~I-1358-304.6/2016).}


\begin{abstract}
We consider sufficient conditions for the existence of $k$-th powers of Hamiltonian cycles in 
	$n$-vertex graphs $G$ with minimum degree $\mu n$ for arbitrarily small~$\mu>0$. About 20 years ago Koml\'os, Sark\"ozy, and Szemer\'edi
	resolved the conjectures of P\'osa and Seymour and obtained optimal minimum degree conditions for this problem  by showing that 
	$\mu=\frac{k}{k+1}$ suffices for large~$n$. For smaller values of $\mu$ the given graph $G$ must satisfy additional 
	assumptions. We show that inducing subgraphs of density $d>0$ on linear subsets of vertices and being inseparable, 
	in the sense that every cut has density at least $\mu>0$, are sufficient assumptions for this problem and, in fact, for a variant 
	of the bandwidth theorem. This generalises recent results of Staden and Treglown. 
\end{abstract}

\maketitle

\section{Introduction}
\label{sec:intro}
We study sufficient conditions for the existence of spanning subgraphs in large finite graphs and begin the discussion with 
powers of Hamiltonian cycles.  For $k\in\NN$ the 
	\emph{$k$-th power} of a given graph~$H$ is the graph~$H^{k}$ on the same vertex set with $xy$ being an edge in $H^k$ if $x$ and $y$ 
	are distinct vertices of~$H$ that are connected in~$H$ by a path of at most~$k$ edges. 
	For simplicity,  we refer to a $k$-th power of a path with at least $k$ vertices as a \emph{$k$-path}. 
	Moreover, we refer to the ordered $k$-tuples of the first and last $k$ vertices of a $k$-path as \emph{ends} of the $k$-path
	and an $(\seq x,\seq y;k)$-path is a $k$-path with ends $\seq x$ and $\seq y$.
	Note that every $k+1$ consecutive vertices of a $k$-path  span a clique and 
	if a graph~$G=(V,E)$ contains \mbox{the~$k$-th} power of a Hamiltonian cycle, it also contains~$\lfloor \tfrac{|V|}{k+1}\rfloor$ pairwise vertex disjoint 
	copies of~$K_{k+1}$ and $G$ contains a~$K_{k+1}$-factor if~$|V|$ is divisible by~$k+1$.
	
	Establishing sufficient conditions for the existence of 
	Hamiltonian cycles in graphs has a long history and Dirac's well known theorem~\cite{dirac1952} yields 
	a best possible minimum degree condition for this problem. The minimum degree of a graph turned out to be an interesting parameter 
	for enforcing a given spanning subgraph and establishing optimal minimum degree conditions for those problems became 
	a fruitful research direction in extremal graph theory (see, e.g.,~\cite{BST09} and the references therein). 
	Already about 50 years ago, the minimum degree 
	problem for $K_{k+1}$-factors was resolved by Corr\'adi and Hajnal~\cite{corradi1963} for $k=2$ and by Hajnal and 
	Szemer\'edi~\cite{hajnal1970} for every $k\geq 3$. P\'osa (see~\cite{E1964}) and Seymour~\cite{seymour1973} asked for 
	a common generalisation of those results on factors and Dirac's theorem and conjectured that the best possible 
	minimum degree conditions for $K_{k+1}$-factors and $k$-th powers of Hamiltonian cycles are the same 
	(given that the number of vertices is divisible by $k+1$). The general conjecture was affirmatively resolved 
	for sufficiently large graphs by Koml\'os, S\'ark\"ozy, and Szemer\'edi~\cite{komlos1998} by establishing the following result.
	\begin{thm}[Koml\'os, Sark\"ozy \& Szemer\'edi 1998]
		\label{thm:KSSz}
		For every positive integer~$k$ there exists~$n_0$ such that if~$G$ is a graph on~$n \geq n_0$ vertices 
		with minimum degree $\delta(G) \geq \frac{k}{k+1} n$,
		then~$G$ contains the~$k$-th power of a Hamiltonian cycle.\qed
	\end{thm}
	Note that for $k=1$ we recover Dirac's theorem (up to the value of $n_0$) and complete and nearly 
	balanced $(k+1)$-partite graphs show that the minimum degree condition in Theorem~\ref{thm:KSSz} 
	is best possible for every~$k$. Those lower bound constructions are ruled out by restricting the 
	independence number of the large graph~$G$. Here we consider the following robust restriction that 
	imposes a uniformly positive edge density for subgraphs induced on linear sized subsets of vertices.
	
	\begin{dfn}
		\label{defn:rhod}	
		We say that a graph $G=(V,E)$ is \emph{$(\rho,d)$-dense} for~$\rho>0$ and $d\in[0,1]$ if
		\[
			e(U)\geq d\frac{|U|^2}{2}-\rho|V|^2
		\]
		for every subset $U\subseteq V$, where $e(U)$ denotes the number of edges of $G$ contained in $U$.
	\end{dfn}
	In fact, it was shown by Staden and Treglown~\cite{staden2018} that such a hereditary density 
	assumption for any~$d>0$ and sufficiently small $\rho>0$
	allows us to reduce the minimum degree condition for $k$-th powers of Hamiltonian 
	cycles to
	\begin{equation}\label{eq:12}
		\delta(G)\geq \Big(\frac{1}{2}+\mu\Big)|V(G)|,
	\end{equation}
	for any $\mu>0$	and sufficiently large vertex sets (see also~\cite{reiher2016} for $K_{k+1}$-factors).
	In particular, the minimum degree condition becomes independent of~$k$. Moreover, the graph~$G$ consisting of two disjoint 
	cliques on close to $n/2$ vertices
	shows that this degree condition 
	is essentially optimal.
	Note that 
	this example is ruled out by the following bipartite version of Definition~\ref{defn:rhod}, which requires 
	\begin{equation}\label{eq:birhod}
		e(X,Y)=\big|\big\{(x,y)\in X\times Y\colon xy\in E(G)\big\}\big|\geq d\,|X||Y|-\rho|V|^2
	\end{equation}
	for all subsets $X$, $Y\subseteq  V$. It was observed by Glock and Joos 
	(see~\cite{staden2018}*{Concluding Remarks}) that imposing property~\eqref{eq:birhod} on $G$  
	allows a further relaxation of the minimum degree condition for $G$ from~\eqref{eq:12} to $\mu |V|$ 
	for arbitrary $\mu>0$. We show that requiring property~\eqref{eq:birhod} for all subsets~$X$ and $Y$ is not needed. 
	It already suffices to assume it only for vertex bipartitions of $G$.
	\begin{dfn}
		\label{defn:muin}	
		We say that a graph $G=(V,E)$ is \emph{$\mu$-inseparable} for some $\mu>0$ if  
		\[
			e(X,V\sm X)\geq \mu\,|X||V\sm X|
		\]
		for every subset 
		$X\subseteq V$.
	\end{dfn}
	Invoking this assumption to subsets $X$ consisting of one vertex only, yields a linear minimum degree for
	$\mu$-inseparable graphs~$G$. It is not hard to show that $\mu$-inseparable graphs are ``well connected'' 
	(see, e.g., Lemma~\ref{lem:muinsep}). Our first result asserts that graphs satisfying the properties of 
	Definitions~\ref{defn:rhod} and~\ref{defn:muin} contain $k$-th 
	powers of Hamiltonian cycles for every fixed integer $k\geq 1$. 

\begin{thm}\label{thm:main}
	For every~$d$,~$\mu \in (0,1]$, and~$k \in \NN$ there exist~$\rho > 0$ and~$n_0$ such that every $(\rho, d)$-dense and $\mu$-inseparable graph~$G$ on~$n \geq n_0$ vertices contains the~$k$-th power of a Hamiltonian cycle.
\end{thm}	
	It is easy to see that every graph $G=(V,E)$ with minimum degree $\delta(G)\geq (1/2+\mu)|V|$
	is $\mu$-inseparable and, consequently, Theorem~\ref{thm:main} strengthens the result of Staden and 
	Treglown for powers of Hamiltonian cycles~\cite{staden2018}. 
	
	Moreover, the $n$-vertex 
	graph~$G$ obtained from two cliques of size $(1/2+\mu/2)n$ which intersect in $\mu n$ vertices is $\mu$-inseparable
    and $(\rho,1/2)$-dense (for any fixed $\rho >0$), while it fails to satisfy property~\eqref{eq:birhod} for arbitrary subsets~$X$ and~$Y$.
    This shows that our result is not covered by the observation of Glock and Joos~\cite{staden2018}*{Concluding Remarks}.
    
	Staden and Treglown and also Glock and Joos not only considered the 
	embedding problem for powers of Hamiltonian cycles, but more generally for bounded degree graphs (with and without 
	small bandwidth, see~\cite{staden2018} for details).  
	We also obtain a generalisation in that direction and establish the following version of the bandwidth theorem from~\cite{BST09}
	for inseparable and uniformly dense graphs (see Theorem~\ref{thm:main2} below). 
	
	We recall that the \emph{bandwidth} $\bw(H)$ of an $n$-vertex graph~$H$ is the 
	maximum distance of two adjacent vertices minimised over all possible orderings of the vertex set of $H$, i.e., 
	\[
		\bw(H)=\min_{\sigma}\max_{xy\in E(H)}\big|\sigma(x)-\sigma(y)\big|\,,
	\]
	where the minimum is taken over all possible bijections $\sigma\colon V(H)\to[n]$. We may refer to an ordering $\sigma$ of $V(H)$
	achieving this minimum $\bw(H)$ as a \emph{bandwidth ordering} of $H$.
\begin{thm}\label{thm:main2}
	For every~$d$,~$\mu\in (0,1]$, and~$\Delta \in \NN$ there exist~$\rho$,~$\beta > 0$ and~$n_0$ such that every $(\rho, d)$-dense and $\mu$-inseparable graph~$G$ 
	on~$n \geq n_0$ vertices contains every $n$-vertex graph~$H$ satisfying $\Delta(H)\leq\Delta$ and $\bw(H)\leq \beta n$.
\end{thm}

\subsection*{Organisation}
The proof of Theorem~\ref{thm:main} utilises the absorption method of 
R\"odl, Ruci\'nski, and Szemer\'edi~\cite{RRSz06}. We discuss this approach in Section~\ref{sec:powham}
and give the details of the proof in Section~\ref{sec:pf-main}. For the proof we use some observations 
on uniformly dense and inseparable graphs, which we collect in Section~\ref{sec:denseInsep}.

Theorem~\ref{thm:main2} follows from Theorem~\ref{thm:main} combined with Szemer\'edi's regularity lemma~\cite{Sz78} and the accompanying blow-up lemma~\cite{blowup97}. Similar reductions appeared in the proofs of the bandwidth theorems in~\cites{BST09,staden2018}.
However, the main challenge in that approach is to deal with the vertices in the exceptional class of the regular partition and 
here we introduce new ideas. We sketch this proof in Section~\ref{sec:bandwidth}.
Finally, in Section~\ref{sec:remarks} we discuss some possible directions for future work.

\section{Uniformly dense and inseparable graphs}
\label{sec:denseInsep}
In this section we shall explore some properties of uniformly dense and inseparable graphs that are crucial for the proof of Theorem \ref{thm:main}. 

\subsection{Properties of uniformly dense graphs}

We start with the following well known fact that uniformly dense graphs contain many cliques of given size.
\begin{lemma}\label{lem:dense}
	For every~$k \in \NN$, $d \in [0, 1]$, and~$\rho > 0$, every~$(\rho, d)$-dense graph~$G$ on $n$ vertices contains at least~$\big( d^{\binom{k}{2}} - (k-1)k\rho \big) n^k$ ordered copies of~$K_k$. 	
\end{lemma}
\begin{proof}
	Let~$G = (V, E)$ be a~$(\rho, d)$-dense graph and~$|V| = n$.
	For~$k = 1$, the assertion is trivial.
	For~$k = 2$, we are counting the number of edges twice. 
	Since $G$ is~$(\rho, d)$-dense, we have~$2|E| \geq 2(d/2 - \rho) n^2$ and the lemma follows.
	
	We continue by induction.
	Let~$k \geq 2$ and assume that for every~$\rho'$, $d'$, it is true that every $(\rho', d')$-dense graph~$H$ contains at least~$\big( d'^{\binom{k}{2}} - (k-1)k\rho' \big) |V(H)|^k$ ordered copies of~$K_k$.
	For counting the ordered copies of~$K_{k+1}$ in~$G$, consider the subset~$ V^* \subseteq V$ of the vertices~$ v \in V$ with~$|N(v)| \geq 1$.
	Let~$\hom(K_{k+1}, G)$ denote the number of ordered copies of~$K_{k+1}$ in~$G$.
	Consequently, we have that
	\[ \hom(K_{k+1},G) = \sum_{v \in V^*} \hom(K_k, G[N(v)]).\]
	
	Since~$G$ is~$(\rho, d)$-dense, for every~$v \in V^*$ and $X\subseteq N(v)$ we have
	\[
		e(X) 
		\geq 
		\frac{d}{2}|X|^2 - \rho n^2 
		= 
		\frac{d}{2}|X|^2 - \rho_v |N(v)|^2\,,
	\]
	for~$\rho_v = \rho n^2 / |N(v)|^2$.
	Thus $G[N(v)]$ is~$(\rho_v, d)$-dense and we can apply the induction hypothesis to get
	\begin{align*}
		\hom(K_{k+1}, G) &\geq \sum_{v \in V^*} \big( d^{\binom{k}{2}} - (k-1)k\rho_v \big) |N(v)|^k \\
		&= d^{\binom{k}{2}} \sum_{v \in V^*}|N(v)|^{k} - (k-1)k\sum_{v \in V^*}\rho_v|N(v)|^k \\
		&\geq d^{\binom{k}{2}} |V^*| \left(\frac{(d - 2\rho) n^{2}}{|V^*|} \right)^k - (k-1)k \sum_{v \in V^*} \frac{\rho n^2}{|N(v)|^2} |N(v)|^{k},
	\end{align*}	
	where the last estimate employed Jensen's inequality and~$\sum_{v \in V^*} N(v) = 2 |E| \geq (d - 2\rho)n^2$. Hence, 
	from~$k \geq 2$ we derive
	\begin{align*}	
		\hom(K_{k+1}, G) 
		&\geq d^{\binom{k}{2}} (d - 2\rho)^k n^{k+1} - (k-1)k\rho n^{k+1} \\
		&\geq d^{\binom{k}{2}} (d^k - 2k\rho) n^{k+1} - (k-1)k\rho n^{k+1} \\
		&\geq \big( d^{\binom{k+1}{2}} - k(k + 1)\rho \big)n^{k+1}\,,
	\end{align*}
	which concludes the proof of the lemma.
\end{proof}

As a corollary, we obtain the following result, which ensures the existence of fairly long $k$-paths in uniformly dense graphs. 
These $k$-paths will be the building blocks for an almost perfect $k$-path cover in the proof of Theorem \ref{thm:main}.
In that proof, we will connect these $k$-paths to an almost spanning $k$-path.
For the connection, it will be convenient to insist that the ends of the $k$-paths are contained in many~$K_{k+1}$'s.
For that we say a clique~$K_{k}$ is $\zeta$-\textit{connectable} in~$G$ if it is contained in at least~$\zeta|V(G)|$ cliques of order~$k+1$.
\begin{cor}[Path Lemma]\label{cor:pathcover}
	For every~$d \in (0, 1]$ and positive integer $k$, there exist~$\rho$, $\zeta > 0$, and $n_0$ such that if~$G$ is a~$(\rho, d)$-dense graph on~$n \geq n_0$ vertices, then~$G$ contains a $k$-path~$P$ with~$\zeta n$ vertices, where every consecutive~$K_k$ in~$P$ is~$\zeta$-connectable. 
\end{cor}

\begin{proof}
	Given~$d \in (0, 1]$ and a positive integer~$k$ we define the constants
	\begin{equation}\label{eq:denseConst}
		\rho = \frac{d^{\binom{k+1}{2}}}{2k(k+1)} \qquad \text{ and } \qquad \zeta = \frac{d^{\binom{k+1}{2}}}{3(k+1)}\,.
	\end{equation}
	
	Let~$G=(V,E)$ be a~$(\rho, d)$-dense graph with~$|V| = n$ sufficiently large.
	Applying Lemma~\ref{lem:dense} and considering the choice of constants in~\eqref{eq:denseConst} show that the number of ordered copies of~$K_{k+1}$ in~$G$ is at least
	\begin{equation}\label{eq:denseEdgeH}
	 \big( d^{\binom{k+1}{2}} - k(k+1)\rho \big) n^{k+1} = \frac{d^{\binom{k+1}{2}}}{2} n^{k+1} \,.
	\end{equation} 
	Define the auxiliary $(k+1)$-uniform hypergraph~$\mathcal{H}_0$ with~$V(\mathcal{H}_0) = V$ and
	\[ E(\mathcal{H}_0)= \left\{ e \in V^{(k+1)} \colon e\text{ spans a $K_{k+1}$ in $G$} \right\}\,. \] 
	
	Successively remove the hyperedges of~$\mathcal{H}_0$ which contain a $k$-tuple that is in at most~$\zeta n$ hyperedges and let~$\mathcal{H}$ be the resulting subhypergraph.
	Note that the number of erased edges is at most
	\[ \binom{n}{k} \cdot \zeta n \overset{\eqref{eq:denseConst}}{<} \frac{d^{\binom{k+1}{2}}}{2(k+1)!} n^{k+1} \overset{\eqref{eq:denseEdgeH}}{\leq} |E(\mathcal{H}_0)| \,.  \]
	Every $k$-tuple of vertices of~$G$ which is contained in some edge of~$\mathcal{H}$ is now~$\zeta$-connectable in~$G$.
	
	Consider \textit{tight paths} in~$\cH$, which are subhypergraphs~$P$ with~$V(P) = \{ x_1, \dots, x_{\ell} \}$ and~$e \in E(P)$ if and only if~$e= \{ x_i, x_{i+1}, \dots, x_{i + k} \}$ for every~$i = 1, \dots, \ell - k$. In particular, consecutive hyperedges in such a path intersect in~$k$ vertices. Observe that any tight path in~$\mathcal{H}$ induces a $k$-path in~$G$ with 
every consecutive~$K_k$ being~$\zeta$-connectable.
	
	Take the longest tight path~$P_0$ in~$\mathcal{H}$. Let~$K$ be the set of the last~$k$ vertices in~$P_0$. 
	If~$e \in E(\mathcal{H})$ is of the form~$e = K \cup \{u\}$ for some~$u \in V$, then~$u$ is already contained in~$P_0$, otherwise the tight path could be enlarged.
	Since every~$k$-tuple contained in some hyperedge of~$\mathcal{H}$ is in at least~$\zeta n$ hyperedges, we know that~$P_0$ has at least~$\zeta n$ vertices. 
\end{proof}

\subsection{Properties of inseparable graphs}
In this section we consider inseparable graphs. First we note that removing a small set of vertices has only little effect on the inseparability.

\begin{lemma}\label{lem:muprime}
	For every~$\mu > 0$ and~$\beta \in (0, 1/2)$, the following holds. If~$G = (V, E)$ is $\mu$-inseparable and~$U \subseteq V$ with~$|U| \leq \beta\mu n$, then~$G[V \setminus U]$ is $(1 - 2\beta)\mu$-inseparable.
\end{lemma}
\begin{proof}
	Suppose by contradiction that~$G[V \setminus U] = (V', E')$ is not $(1 - 2\beta)\mu$-inseparable. Thus, there exists~$X \subseteq V'$ with $|X|\leq n/2$ such that~$e(X, V' \setminus X) < (1 - 2\beta)\mu |X||V' \setminus X|$. 
	Consider the partition of~$V$ into the sets~$X$ and~$(V' \setminus X) \cup U = V \setminus X$. We have that
	\begin{align}
	e(X, V \setminus X) &< (1 - 2\beta)\mu |X||V' \setminus X| + |U||X| \nonumber \\
	&= (1 - 2\beta)\mu |X|(|V| - |U| - |X|) + |U||X| \nonumber\\
	&= \mu |X||V \setminus X| -2\beta\mu |X||V \setminus X|  + ( 1 - ( 1 - 2\beta)\mu)|U||X| \,. \label{eq:muprime}
	\end{align}
	Since $|V \setminus X| \geq n /2$,~$\beta \mu n \geq |U|$, and~$\beta < 1/2$ we have
	\[2\beta\mu |X||V \setminus X| \geq \beta\mu n |X| \geq |U||X| \geq ( 1 - ( 1 - 2\beta)\mu)|U||X|\,.\]
	Together with \eqref{eq:muprime}, we derive that~$e(X, V \setminus X) < \mu |X||V \setminus X|$, which contradicts the assumption that~$G$ is~$\mu$-inseparable.
\end{proof}

The key property of inseparable graphs is that between any pair of vertices there exist many paths of bounded length.

\begin{lemma}\label{lem:muinsep}
	For every $\mu \in (0,1]$, there exist $c > 0$ and integers $L$, $n_0$ such that every $\mu$-inseparable graph~$G = (V, E)$ on~$|V| = n \geq n_0$ vertices satisfies the following.
	
	For every two distinct vertices~$x$,~$y \in V$, there is some integer~$\ell$ with~$0 \leq \ell \leq L$ such that 
	the number of $(x, y)$-walks with~$\ell$ inner vertices in~$G$ is at least~$cn^{\ell}$.
\end{lemma}

\begin{proof}
	Given $\mu$ we define
	\begin{equation}\label{eq:muinsepCons}
		 L = \left\lfloor \frac{8}{\mu} \right\rfloor, \qquad \delta_{i} = \left( \frac{\mu^2}{3} \right)^{i}\left( \frac{1}{2} \right)^{\binom{i+1}{2}}, \qquad \text{ and } \quad c = \frac{\mu^2}{48} \delta_{\lfloor 4/ \mu \rfloor}^2\,.
	\end{equation}
	
	Let~$G$ be a sufficiently large~$\mu$-inseparable graph on~$n$ vertices and~$x$,~$y$ be two distinct vertices of~$G$. Consider for each~$i \geq 0$ the set of vertices~$v$ that can be reached from~$x$ by ``many'' walks in~$G$ with~$i$ inner vertices. For that we define
	\begin{align*}
		X_i = \left\{ v \in V\colon \text{there are~$\delta_i n^i$  $(x, v)$-walks with~$i$ inner vertices}  \right\} \quad \text{ and } \quad X^i = \bigcup_{0\leq j \leq i} X_j\,.
	\end{align*}
	Analogously, consider the vertices~$v$ that can be reached from~$y$ by~$\delta_i n^i$ walks in~$G$ with~$i$ inner vertices and define the sets~$Y_i$ and~$Y^i$ in the same way.
	
	Observe that~$X_0 = X^0 = N(x)$ and since~$G$ is~$\mu$-inseparable,~$|N(x)| \geq \mu (n -1)$. Moreover,~$X^i \subseteq X^{i+1}$ and we shall show that as long as~$|X^{i}|$ is not too large, then~$|X^{i+1}|$ is substantially larger than~$|X^i|$.
	More precisely, we show for every~$i\geq 0$ that
	\begin{equation}\label{eq:muinsepIncrease}
		|X^i| \leq \frac{2}{3}n \quad \implies \quad |X^{i+1} \setminus X^i| \geq \frac{\mu}{6} n \,.
	\end{equation}  
	
	Before verifying \eqref{eq:muinsepIncrease}, we conclude the proof of Lemma~\ref{lem:muinsep}. In fact,~\eqref{eq:muinsepIncrease} implies that there is some~$i_0 < \lfloor 4/ \mu \rfloor$ such that~$|X^{i_0}| > 2n/3$. Applying the same argument for~$Y^i$, we get some~$j_0 < \lfloor 4/ \mu \rfloor$ such that~$|Y^{j_0}| > 2n /3$ and, hence,~$|X^{i_0} \cap Y^{j_0}| \geq n/3$.
	
	Each vertex~$ v \in X^{i_0} \cap Y^{j_0}$ can be used to create many $(y,x)$-walks with possibly different number of inner vertices. However, by the pigeonhole principle there are integers~$a$,~$b$ with~$0 \leq a \leq i_0$ and~$0 \leq b \leq j_0$ such that
	\begin{equation}\label{eq:muinsepInters}
		|X_a \cap Y_b| \geq \frac{|X^{i_0} \cap Y^{j_0}|}{(i_0 + 1)(j_0 + 1)} \geq \frac{\mu^2 n}{48}\,.
	\end{equation}
		
	For each~$v \in X_a \cap Y_b$ there exist~$\delta_{a}n^{a}$ $(x, v)$-walks and~$\delta_{b}n^{b}$ $(v, y)$-walks with~$a$ and~$b$ inner vertices, respectively.
	Concatenating these walks leads to at least
	\[ \delta_{a} \delta_{b} n^{a + b} \cdot |X_a \cap Y_b| \] 
	$(x, y)$-walks, with~$\ell = a + b + 1$ inner vertices.
	The choice of constants in~\eqref{eq:muinsepCons} and~\eqref{eq:muinsepInters} conclude the proof.
	
	It is left to verify \eqref{eq:muinsepIncrease}. Suppose~$|X^i| \leq 2n/3$ and consider the complement~$Z = V \setminus X^i$. 
	Owing to the~$\mu$-inseparability of~$G$ we have
	\begin{equation}\label{eq:muinsepLowEdges}
	e(X^i, Z) \geq \mu|X^i||Z| \,.		
	\end{equation}
	
	Note that each vertex~$v$ with at least~$\delta_{j+1} n / \delta_j$ neighbours in~$X_j$ belongs to~$X_{j+1}$.
	Since~$Z$ is disjoint from~$X^i$, we have
	\begin{equation}\label{eq:muinsepUpEdges1}
		e(X^{i - 1}, Z) < |Z| \cdot \sum_{j = 0}^{i - 1} \frac{\delta_{j+1}}{\delta_j} n \,.
	\end{equation}
	Moreover, supposing by contradiction that~\eqref{eq:muinsepIncrease} fails, we also have
	\begin{equation} \label{eq:muinsepUpEdges2}
	e(X_i, Z) < |Z| \cdot \frac{ \delta_{i+1}}{\delta_i}n + \frac{\mu}{6} |X_i| n \,. 
	\end{equation}
	
	Combining \eqref{eq:muinsepUpEdges1} and \eqref{eq:muinsepUpEdges2} we arrive at
	\begin{equation}\label{eq:muinsepUpEdges}
		e(X^i, Z) < |Z| \cdot \sum_{j = 0}^{i-1}\frac{\delta_{j+1}}{\delta_j} n + |Z|\cdot \frac{ \delta_{i+1}}{\delta_i}n + \frac{\mu}{6} n |X_i| =
		|Z| \cdot \sum_{j = 0}^{i} \frac{\delta_{j+1}}{\delta_j} n +  \frac{\mu}{6} n |X_i| \,. 
	\end{equation}
	Owing to the choice of~$\delta_j$ in \eqref{eq:muinsepCons} we have 
	\begin{align*}
		\sum_{j = 0}^{i} \frac{\delta_{j+1}}{\delta_j} = \frac{\mu^2}{3} \sum_{j = 0}^{i} \left( \frac{1}{2} \right)^{j+1} \leq \frac{\mu^2}{3} \,.
	\end{align*}
	Furthermore, since~$ |X^i| \geq |X^0| = |N(x)| \geq \mu(n-1) $ and~$|Z| = |V \setminus X^i| \geq n/3$, we derive for sufficiently large~$n$ from \eqref{eq:muinsepUpEdges} that
	\[ e(X^i, Z)< \frac{\mu^2}{3}|Z|n + \frac{\mu}{6}|X_i|n \leq \frac{\mu}{2}|Z||X^i| + \frac{\mu}{2}|X_i||Z| \leq \mu|X^i||Z|\,, \]
	which contradicts \eqref{eq:muinsepLowEdges}.
\end{proof}

\section{Absorption method and powers of Hamiltonian cycles}
\label{sec:powham}
	
The proof of Theorem~\ref{thm:main} is based on the absorption method and follows the strategy from~\cite{RRSz06}. Roughly speaking, this method 
splits the problem of finding a $k$-th power of a Hamiltonian cycle into the following three parts:
	\begin{enumerate}
		\item finding an almost perfect cover with only ``few'' $k$-paths, 
	 	\item ensuring the abundant existence of so-called \textit{absorbers}, and
		\item connecting those absorbers and paths to an almost spanning $k$-th power  of a cycle.
	\end{enumerate}
	The first part is achieved by Corollary~\ref{cor:pathcover} and only makes use of the $(\rho,d)$-denseness of~$G$.
	For the second part of the absorption method again the $(\rho,d)$-denseness of~$G$ suffices. However, for the connection of 
	these absorbers the $\mu$-inseparability is required. The appropriate \emph{connecting lemma}, which is also utilised for connecting 
	the paths of the almost perfect cover from the first part, is given in Section~\ref{subsec:connecting}.
	In Section~\ref{sec:apl} we establish the \emph{absorbing path lemma} and in Section~\ref{sec:pf-main} we combine these results and 
	deduce Theorem~\ref{thm:main}.
	
\subsection{Connecting Lemma}
\label{subsec:connecting}
The connecting lemma asserts that any two connectable $K_k$'s in a uniformly dense and inseparable graph $G$ are connected by ``many'' $k$-paths of bounded length. As shown in Lemma~\ref{lem:muinsep}, for $k=1$ this is a direct consequence of the inseparability. For $k\geq 2$ we combine 
Lemma~\ref{lem:muinsep} with Lemma~\ref{lem:dense} by a standard supersaturation argument to obtain the desired $k$-paths.

\begin{lemma}[Connecting Lemma]\label{lem:conn}
	For every~$d$, $\mu\in(0,1]$, $\zeta > 0$, and every integer~$k \geq 1$, there exist~$\rho$,~$\xi  > 0$ and integers~$M$, $n_0\in\NN$ such that every $(\rho, d)$-dense and $\mu$-inseparable graph~$G=(V,E)$ on~$|V|=n \geq n_0$ vertices satisfies the following.
	
	For every pair~$\vec x$, $\vec y\in V^k$ of disjoint $\zeta$-connectable~$K_k$ in~$G$, there is some integer $m\leq M$ such that the number of  $(\vec x, \vec y; k)$-paths with~$m$ inner vertices in~$G$ is least~$\xi n^m$.
\end{lemma}	

\begin{proof}
	Given~$d$, $\mu$, $\zeta$ and~$k$, we shall fix constants~$\rho$, $\xi$, $M$ and~$n_0$. For that we first 
	apply Lemma~\ref{lem:muinsep} for~$\mu$ and obtain~$L$ and $c$. Next we define auxiliary constants~$\xi_i$ for integers~$i \geq 0$ inductively through
	\begin{equation}\label{eq:connConst1}
	\xi_0 = \frac{\zeta^2 c}{L + 1} \qquad \text{ and } \qquad \xi_{i+1} =\frac{d^{ \binom{k}{2} }}{2 k!} \left( \frac{\xi_i}{2} \right)^{k+1} \,.  
	\end{equation}
	Finally, we set
	\begin{equation}\label{eq:connConst2}
		\xi = \frac{\xi_{L+2}}{2}, \qquad \rho = \frac{d^{ \binom{k}{2} } \xi^2}{8 k^2}, \qquad \text{ and } \qquad M = (L + 2)k\,,
	\end{equation}
	and let~$n$ be sufficiently large.
	
	Let~$G$ be a~$(\rho, d)$-dense and~$\mu$-inseparable graph on~$n$ vertices and let~$\vec x = (x_1, \dots, x_k)$ and~$\vec y = (y_1, \dots, y_k )$ be two disjoint~$\zeta$-connectable $K_k$ in~$G$. 
	
	We consider the following type of graphs that will be useful to obtain the desired $(\vec x, \vec y; k)$-paths. 
	For integers~$k \geq 1$,~$\ell \geq 0$ and~$0 \leq a \leq \ell$, a graph~$R$ is a~$(k, \ell, a)$-\textit{rope} if it can be obtained from a path on~$\ell + 2$ vertices by blowing up the first, the last, and the first~$a$ inner vertices into~$K_k$. More precisely, the vertex set of~$R$ is
	\[V(R) = Z_0 \cup \dots \cup Z_{\ell + 1}\]
	such that
	\[|Z_0| = |Z_{\ell + 1}| = k = |Z_1| = \dots = |Z_{a}| \quad \text{ and } \quad |Z_{a+1}| = \dots = |Z_{\ell}| = 1\,.\]
	The edges of~$R$ are such that~$Z_0, Z_1, \dots, Z_a$, and $Z_{\ell + 1}$ each induce a~$K_k$ and between any consecutive pair~$(Z_i, Z_{i+ 1})$, for~$i = 0, \dots, \ell$, all~$|Z_i||Z_{i+ 1}|$ edges are present. Note that we do not insist that the sets~$Z_i$ are pairwise disjoint. If the vertices in~$Z_0$ are those of~$\vec x$ and the vertices in~$Z_{\ell + 1}$ are those of~$\vec y$, then the rope is said to be a~$(\vec x, \vec y; k, \ell, a)$-rope and the sets~$(Z_1, \dots,Z_a)$ are called the \textit{inner parts} of the rope. 

	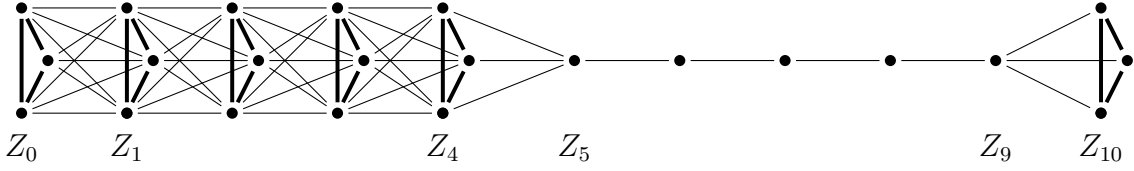
\begin{figure}[ht]
		\begin{tikzpicture}[scale = .7]
		
		\fill (0,1) circle (3pt) node[below = 4pt]{$Z_0$};
		\node (x1) at (0,1) {};
		\fill (0.5,2) circle (3pt);
		\node (x2) at (0.5,2) {};
		\fill (0,3) circle (3pt);
		\node (x3) at (0,3) {};
		\path[draw, line width =1.5pt] (x1) -- (x2);
		\path[draw, line width =1.5pt] (x2) -- (x3);
		\path[draw, line width =1.5pt] (x1) -- (x3);
		
		\foreach \x in {2, 4, 6, 8}
		{
			\node (x4) at (x1) {};
			\node (x5) at (x2) {};
			\node (x6) at (x3) {};
			
			\fill (\x,1) circle (3pt);
			\node (x1) at (\x,1) {};
			\fill (\x + 0.5,2) circle (3pt);
			\node (x2) at (\x + 0.5,2) {};
			\fill (\x,3) circle (3pt);
			\node (x3) at (\x,3) {};
			\path[draw, line width =1.5pt] (x1) -- (x2);
			\path[draw, line width =1.5pt] (x2) -- (x3);
			\path[draw, line width =1.5pt] (x3) --(x1);
			\foreach \y in {(x1), (x2), (x3)}{
				\path[draw] \y -- (x4);
				\path[draw] \y -- (x5);
				\path[draw] \y -- (x6);
			}
		} 
	
		\fill (10.5,2) circle (3pt);
		\node (x4) at (10.5,2) {};
		\path[draw] (x1) -- (x4);
		\path[draw] (x2) -- (x4);
		\path[draw] (x3) -- (x4);
		
		\foreach \x in {12.5, 14.5, 16.5, 18.5}
		{
			\node (x1) at (x4) {};
			\fill (\x,2) circle (3pt);
			\node (x4) at (\x,2) {};
			\path[draw] (x1) -- (x4);
		}
	
		\fill (20.5,1) circle (3pt);
		\node (x1) at (20.5,1) {};
		\fill (21,2) circle (3pt);
		\node (x2) at (21,2) {};
		\fill (20.5,3) circle (3pt);
		\node (x3) at (20.5,3) {};
		\path[draw, line width =1.5pt] (x1) -- (x2);
		\path[draw, line width =1.5pt] (x2) -- (x3);
		\path[draw, line width =1.5pt] (x1) -- (x3);
		\path[draw] (x4) -- (x1);
		\path[draw] (x4) -- (x2);
		\path[draw] (x4) -- (x3);
		
		\fill (2,1) circle (3pt) node[below = 4pt]{$Z_1$};
		\fill (8,1) circle (3pt) node[below = 4pt]{$Z_4$};
		\fill[white] (10.5,1) circle (3pt) node[black, below = 4pt]{$Z_5$};
		\fill[white] (18.5,1) circle (3pt) node[black, below = 4pt]{$Z_9$};
		\fill (20.5,1) circle (3pt) node[black, below = 4pt]{$Z_{10}$};
		
		\end{tikzpicture}
		
		\caption{A $(3, 9, 4)$-rope.}
	\end{figure}
	
	We shall prove the following assertion for fixed cliques~$\vec x$ and~$\vec y$.
	\begin{clm}\label{clm:ConnRope}
		There exists~$\ell \leq L + 2$ such that for every~$a = 0, \dots, \ell$ there are~$\xi_{a}n^{ak + (\ell - a)}$ $(\vec x, \vec y; k, \ell, a)$-\textit{ropes} in~$G$. 
	\end{clm}
	Note that, for~$a = 0$, Claim~\ref{clm:ConnRope} ensures many walks between~$N(\vec x)$ and~$N(\vec y)$, which indeed are provided by Lemma~\ref{lem:muinsep}.	
	For~$a = \ell$, it is easy to see that a~$(\vec x, \vec y; k, \ell, \ell)$-rope (without any vertex repetition) contains a~$(\vec x, \vec y; k)$-path with~$m = \ell  k$ inner vertices. 
	The number of~$(\vec x, \vec y; k, \ell, \ell)$-ropes with vertex repetitions is bounded by~$m^2n^{m - 1}$. Excluding these ropes from those obtained by Claim~\ref{clm:ConnRope} for~$a = \ell$ yields at least~$\xi_{\ell} n^m - m^2n^{m - 1} \geq \xi_{\ell} n^m /2 \geq \xi n^m$ $(\vec x, \vec y; k)$-paths, for sufficiently large~$n$.
	Thus, for~$a = \ell$, the claim leads to the conclusion of Lemma~\ref{lem:conn} and it is left to verify the claim.
	\end{proof}
	\begin{proof}[Proof of Claim \ref{clm:ConnRope}]
	First we fix the integer~$\ell$.
	Consider the neighbourhoods~$N(\vec x)$ and~$N(\vec y)$. 
	Since~$\vec x$ and~$\vec y$ are~$\zeta$-connectable, we have~$|N(\vec x)|$,~$|N(\vec y)| \geq \zeta n$.
	Since~$G$ is~$\mu$-inseparable, for each pair of distinct vertices~$(x, y) \in N(\vec x) \times N(\vec y)$, by Lemma~\ref{lem:muinsep}, there are at least~$cn^{\ell(x, y)}$ $(x, y)$-walks with~$\ell(x, y)$ inner vertices and~$0 \leq \ell(x, y) \leq L$. Hence, by the pigeonhole principle, there is some~$\ell$ with~$2 \leq \ell \leq L + 2$ such that there are at least~$\zeta^2 c n^{\ell} / (L+1) = \xi_0 n^{\ell}$ $(x_k, y_1)$-walks with exactly~$\ell$ inner vertices and such that the first and last inner vertex is from~$N(\vec x)$ and~$N(\vec y)$, respectively. 
	This yields the claim for~$a = 0$. 
	
	We proceed in an inductive manner. Let~$a \geq 0$ and assume by induction that~$G$ contains at least~$\xi_{a}n^{ak + (\ell - a)}$ $(\vec x, \vec y;k, \ell, a)$-ropes. For many of such ropes we now blow up the $1$-element part~$Z_{a+1}$. Consider collections~$\mathcal{Z} = (Z_1, \dots, Z_a, Z_{a+2}, \dots, Z_{\ell})$ and sets~$U_{\mathcal{Z}}$ such that~$u \in U_{\mathcal{Z}}$ if and only if
	\[(Z_1, \dots, Z_a, \{u\}, Z_{a+2}, \dots, Z_{\ell})\]
	are the inner parts of a~$(\vec x, \vec y;k, \ell, a)$-rope. By a standard averaging argument, it is easy to see that there exist~$\xi_a n^{ak + (\ell - a) - 1} / 2$ collections~$\mathcal{Z}$ such that the set~$U_{\mathcal{Z}}$ has size at least~$\xi_a n /2$.
	
	Note that each~$K_k$ contained in~$U_{\mathcal{Z}}$ yields a~$(\vec x, \vec y; k, \ell, a+1)$-rope. Lemma~\ref{lem:dense} applied to every set~$U_{\mathcal{Z}}$ of size at least~$\xi_a n /2$ leads to
	\begin{equation*}
		\frac{1}{k!} \left( d^{\binom{k}{2}} - (k-1)k\frac{\rho n^2}{|U_{\mathcal{Z}}|^2} \right)|U_{\mathcal{Z}}|^k 
		\overset{\eqref{eq:connConst2}}{\geq} 
		\frac{d^{\binom{k}{2}}}{2k!}\left(\frac{\xi_a}{2}\right)^kn^k
	\end{equation*}
	unordered copies of~$K_k$ in a given~$U_{\mathcal{Z}}$ of size at least~$\xi_a n /2$ .
	Hence, there are at least
	\begin{equation*}
		 \frac{\xi_{a}}{2}n^{ak + (\ell - a) -1} \cdot \frac{d^{\binom{k}{2}}}{2k!} \left(\frac{\xi_a }{2}\right)^{k} n^{k} \overset{\eqref{eq:connConst1}}{=} \xi_{a+1}n^{(a+ 1)k + (\ell - (a+ 1))}
	\end{equation*}
	$(\vec x, \vec y;k, \ell, a + 1)$-ropes in~$G$, which concludes the proof of Lemma~\ref{lem:conn}. 
	\end{proof}

\subsection{Absorbing Path Lemma}
\label{sec:apl}
For a given vertex $v$  a clique $K_{2k}$ contained in $N(v)$ can be used as an absorber for~$v$, in the sense that obviously 
the $K_{2k}$ induces a $k$-path on~$2k$ vertices. Moreover, placing $v$ in the middle of this $K_{2k}$ yields a $k$-path 
containing $v$ that starts and ends with the same $K_k$'s. Since inseparable and  uniformly dense graphs $G=(V,E)$ 
have minimum degree linear in the number of vertices, by Lemma~\ref{lem:dense}
the uniform density yields
many cliques $K_{2k}$ in the neighbourhood of any given vertex $v$. Moreover, for many of these $K_{2k}$ all the
$K_k$'s contained in it are connectable and, hence, together with the connecting lemma we can build an \emph{absorbing path} defined below.

\begin{dfn}
	For a graph~$G = (V, E)$ on~$n$ vertices, an integer~$k\geq 1$, and~$\alpha \geq 0$, we say that a~$(\vec x, \vec y; k)$-path~$P$ in~$G$ is~$\alpha$-\textit{absorbing} if for every set~$X \subseteq V \setminus V(P)$ of size~$|X| \leq \alpha n$, there is a~$(\vec x, \vec y; k)$-path~$Q$ in~$G$ with vertex set~$V(Q) = V(P) \cup X$.  
\end{dfn}

\begin{lemma}[Absorbing Path Lemma]\label{lem:abspath}
	For every~$d$,~$\mu \in (0,1]$, and integer~$k \geq 1$, there exist~$\rho$,~$\zeta$,~$\alpha > 0$, with~$\zeta \leq \mu/2$, and~$n_0$ such that every $(\rho, d)$-dense and $\mu$-inseparable graph~$G$ on~$n \geq n_0$ vertices contains an~$\alpha$-absorbing~$(\vec x, \vec y; k)$-path~$P_A$ of size~$|V(P_A)| \leq \zeta n/2$, and~$\vec x$,~$\vec y$ being $\zeta$-connectable. 
\end{lemma}

\begin{proof}
	Given~$d$,~$\mu$, and~$k$, we set
	\begin{equation}\label{eq:absConst}
		\zeta = \frac{d^{\binom{2k +1}{2}} \mu^{2k +1} }{2^{2k + 3}}\,.
	\end{equation}
	Applying Lemma~\ref{lem:conn} for~$d$,~$\mu / 2$, and~$\zeta / 2$ yields constants~$\rho'$,~$\xi$,~$M$, and~$n_0'$ and we fix
	\begin{equation}\label{eq:absConst2}
		\alpha =  \frac{\zeta^2}{24(10k^2 + M)}  \qquad \text{and} \qquad \rho = \min \left\{ \frac{\rho'}{4}, \frac{ d^{\binom{2k + 1}{2}} \mu^2 }{8 (2k + 1)^2 } \right\} \,,
	\end{equation}
	and let a sufficiently large $(\rho, d)$-dense and $\mu$-inseparable graph~$G = (V, E)$ on~$n$ vertices be given.
	
	For every vertex~$v \in V$, call an ordered~$K_{2k}$ contained in~$N(v)$ with ordered vertex 
	set~$(x_1, \dots, x_{2k})$ an $v$-\textit{absorber}, if~$\vec x_{v} = (x_1, \dots, x_k)$ and~$\vec y_v = (x_{k+1}, \dots, x_{2k})$ 
	are $\zeta$-connectable (ordered) $K_k$'s in $G$.
	Note that $(x_1, \dots, x_{2k})$ and~$(x_1, \dots, x_k, v, x_{k+1}, \dots, x_{2k})$ are both $(\vec x_v, \vec y_v; k)$-paths with $\zeta$-connectable ends. We denote by $\cA_v\subseteq V^{2k}$ the set of all $v$-absorbers and we let 
	\[
		\cA=\bigcup_{v\in V} \cA_v
	\]
	be the set of all absorbers in $G$.
	
	The~$\alpha$-absorbing~$(\vec x, \vec y; k)$-path~$P_A$ is constructed by considering 
	a collection~$A\subseteq \cA$ independently at random.
	We shall show that a.a.s.\ for every vertex~$v$ the collection~$A$ will contain 
	``many''~$v$-absorbers from $\cA_v$. After erasing intersecting absorbers, 
	we shall connect the remaining ones to a path~$P_A$ by repeated applications of Lemma~\ref{lem:conn}.
	
	First we prove the existence of ``many'' $v$-absorbers 
	for every~$v \in V$ in~$G$. Since~$G$ is $\mu$-inseparable, we have~$|N(v)| \geq \mu (n-1) \geq \mu n / 2$ for sufficiently large~$n$. Consequently, the induced subgraph~$G[N(v)]$ is $(4\rho / \mu^2, d)$-dense and Lemma~\ref{lem:dense} shows that the number of ordered~$K_{2k + 1}$ in~$G[N(v)]$ is at least
	\[
	\Big( d^{\binom{2k + 1}{2}} - (2k)(2k+1) \frac{4 \rho}{\mu^2} \Big) |N(v)|^{2k+ 1} \overset{\eqref{eq:absConst2}}{\geq} \frac{d^{\binom{2k + 1}{2}} }{2} \cdot \Big( \frac{\mu} {2} \Big)^{2k+1} \cdot n^{2k+1} \overset{\eqref{eq:absConst}}{=} 2\zeta n^{2k+1}\,.
	\]
	Since there are at most~$n^{2k}$ different~$K_{2k}$ and each~$K_{2k}$ is contained in at most~$n$ different~$K_{2k+1}$, by a simple averaging argument, there exist at least~$\zeta n^{2k}$ different~$K_{2k}$ each contained in~$\zeta n$ different~$K_{2k + 1}$ in~$G[N(v)]$. Consequently, $|\cA_v|\geq \zeta n^{2k}$ for every~$v \in V$.	
	
	Set
	\begin{equation}\label{eq:absProb}
		p = \frac {\zeta}{6(10k^2 + M) n^{2k -1}}\,.
	\end{equation}
	and consider a random collection~$A\subseteq \cA$, in which every ordered~$K_{2k}\in\cA$ is included independently with probability~$p$.
	Let~$X_v$ be the random variable $|A\cap\cA_v|$. Then,
	\begin{align*}
	\EE X_v \geq \zeta n^{2k}p \overset{\eqref{eq:absProb} }{=} \frac{\zeta^2 n}{6(10k^2 + M)}\,.
	\end{align*}
	
	Since~$X_v$ is binomially distributed, by the union bound and Chernoff's inequality (see, e.g., \cite{JLR}*{Theorem 2.1}), we have
	\begin{align}\label{eq:absAbs}
	\PP  \Big( \exists v \colon X_v \leq \frac{\zeta^2 n}{12(10k^2 + M)} \Big) &\leq n \cdot \max_{v \in V} \PP \bigg( X_v \leq \frac{\zeta^2 n}{12(10k^2 + M)} \bigg) \nonumber \\
	& \leq n \cdot \exp \bigg( -\frac{\zeta^2 n}{48(10k^2 + M)} \bigg) < \frac{1}{3}\,, 
	\end{align}
	for sufficiently large~$n$.
	
	Consider now the pairs of absorbers in~$A$ that share some vertex. Let~$Y$ be the random variable that counts the number of such intersecting pairs.
	There are at most~$5k^2n^{4k - 1}$ possible intersecting pairs in $\cA$ and, hence,~$\EE Y < 5k^2n^{4k - 1} p^2$. Markov's inequality yields
	\begin{equation}\label{eq:absInter}
	\PP \Big(Y \geq \frac{\zeta^2 n}{24(10k^2 + M)} \Big) \overset{\eqref{eq:absProb} }{\leq} \PP (Y \geq 15 k^2 n^{4k - 1} p^2) \leq \frac{ \EE Y}{15 k^2 n^{4k - 1} p^2} < \frac{1}{3}\,.
	\end{equation}
	
	For the size of~$A$ we note that~$\EE |A| < n^{2k}p$ and another application of Markov's inequality shows
	\begin{equation}\label{eq:absSize}
		\PP \Big(|A| \geq \frac{\zeta n}{2(10k^2 + M)} \Big) \overset{\eqref{eq:absProb}}{=} \PP (|A| \geq 3n^{2k}p) \leq \frac{\EE |A|}{3 n^{2k}p} < \frac{1}{3}\,.
	\end{equation}
	
	Thus, by \eqref{eq:absAbs}, \eqref{eq:absInter}, and \eqref{eq:absSize}, there exists a collection~$A_0\subseteq \cA$ satisfying
	\begin{enumerate}[label=(\roman*)]
	\item $|A_0\cap\cA_v|\geq \zeta^2 n / (12(10k^2 + M)$ for every~$v \in V$,
	\label{prop:absAbs}
	\item there are at most~$\zeta^2 n / (24(10k^2 + M))$ pairs of intersecting absorbers in~$A_0$, 
	\label{prop:absInter}
	\item the size~$|A_0|$ is at most~$\zeta n / (2(10k^2 + M))$. \label{prop:absSize}
	\end{enumerate}
	From each pair of intersecting absorbers, delete one of them in an arbitrary way and let~$A_1 = \{ \hat{K}_1, \dots, \hat{K}_{m} \} \subseteq A_0$ be the set of absorbers in~$A_0$ obtained this way. It follows from properties~\ref{prop:absAbs} and~\ref{prop:absInter}
	that 
	\[
		\big|A_1\cap \cA_v\big|
		\geq \frac{\zeta^2 n}{24(10k^2 + M)}
		= \alpha n
	\]
	for every~$v \in V$, i.e., $A_1$ contains at least $\alpha n$ $v$-absorber for every vertex $v\in V$. 
	
	In the final step we connect the absorbers in~$A_1$ to a $k$-path~$P_A$. We construct~$P_A$ inductively by repeated applications of Lemma~\ref{lem:conn}. 
	Suppose we already obtained a path~$P^i$ that contains~$\hat{K}_1, \dots, \hat{K}_i$ and 
	\begin{align*}
		|V(P^i)| \leq i \cdot 2k + (i - 1)M\,,
	\end{align*}
	below we establish the existence of~$P^{i+1}$ that in addition contains~$\hat{K}_{i+1}$ and satisfies  
	\begin{align*}
		|V(P^{i+1})| \leq (i + 1)\cdot 2k + i \cdot M\,.
	\end{align*}
	For that, set
	\[ Z^i = V(P^i) \cup \bigcup_{j = i+1}^{m} V(\hat{K}_j)
	\]
	and observe that for every~$i \leq m$,
	\begin{align}\label{eq:absPSize}
	|Z^i| &\leq i \cdot 2k + (i - 1)M + (m - i) \cdot 2k < m \cdot (2k + M)  \nonumber \\
	&\overset{\ref{prop:absSize}}{\leq} \frac{\zeta n}{2(10k^2 + M)} (2k + M) \leq \frac{\zeta n}{2} \overset{\eqref{eq:absConst}}{\leq} \frac{\mu n}{4}\,.
	\end{align}
	
	Let~$\vec x$ be the last~$k$ vertices of~$\hat{K}_i$ and let~$\vec y$ be the first~$k$ vertices of~$\hat{K}_{i+1}$.
	In view of the choice of constants in~\eqref{eq:absConst2}, the observation~\eqref{eq:absPSize}, and Lemma~\ref{lem:muprime}, the induced subgraph~$G' = G[(V\setminus Z^i) \cup V(\vec x) \cup V(\vec y)]$ is $(\rho', d)$-dense and $(\mu / 2)$-inseparable, and~$\vec x$ and~$\vec y$ are~$\zeta/2$-connectable in~$G'$. Consequently, there is an~$(\vec x,\vec y;k)$-path in~$G'$ with at most~$M$ inner vertices outside~$Z^i$. Together with~$P^i$ this yields~$P^{i+1}$ with the desired properties.
	
	Since~$P^m$ contains at least $\alpha n$ distinct $v$-absorbers for every $v\in V$, it is an~$\alpha$-absorbing path. Moreover, the first~$k$ vertices in~$P^m$ are from~$\hat{K}_1$ and the last~$k$ vertices are from~$\hat{K}_{m}$, which by definition are $\zeta$-connectable cliques in~$G$, which shows that~$P_A = P^m$ has the desired properties. 
\end{proof}

\subsection{Proof of the main result}\label{sec:pf-main}
Having established the Connecting Lemma (Lemma~\ref{lem:conn}), the Absorbing Path Lemma (Lemma~\ref{lem:abspath}), and the Path Lemma (Corollary~\ref{cor:pathcover}), we are ready to deduce Theorem~\ref{thm:main}.

\begin{proof}[Proof of Theorem \ref{thm:main}]
	The proof of Theorem~\ref{thm:main} is based on the absorption method and we start by fixing all involved constants.
	Given~$d$,~$\mu$, and~$k$, applying Lemma~\ref{lem:abspath} (Absorbing Path Lemma) yields constants~$\rhoa$,~$\zetaa$,~$\alphaa$, and an application of Corollary~\ref{cor:pathcover} (Path Lemma) yields~$\rhop$ and~$\zetap$. For an application of Lemma~\ref{lem:conn} (Connecting Lemma), set
	\begin{equation}\label{eq:mainZeta}
	\zetac = \min \Big\{ \frac{\zetaa}{2}, \frac{\alphaa \zetap }{2}\Big\}\,,
	\end{equation}
	and apply the Connecting Lemma with~$\mu /2$,~$d$,~$\zetac$, and~$k$ to attain constants~$\rhoc$,~$\xic$, and~$\mc$. Finally, set
	\begin{equation}\label{eq:mainRho}
	\rho = \min \Big\{ \rhoa, \frac{\alphaa^2 \rhop}{4}, \frac{\rhoc}{4}\Big\}\,,
	\end{equation}
	and let~$n$ be sufficiently large. In particular, we may assume  that
	\begin{equation}\label{eq:mainN}
		\frac{2}{\alphaa \zetap} \mc^2 < \frac{\xic}{4} \Big( \frac{\alphaa}{8}\Big)^{\mc} n\,.
	\end{equation}
	
	Let~$G = (V, E)$ be a $(\rho, d)$-dense and $\mu$-inseparable graph on~$n$ vertices.
	By the Absorbing Path Lemma, there is an~$\alphaa$-absorbing $(\vec x_A, \vec y_A; k)$-path~$P_A$ contained in~$G$ with~$|V(P_A)| \leq \zetaa n/2$ and~$\vec x_A$,~$ \vec y_A$ being~$\zetaa$-connectable in~$G$. 
	This path will be set aside and with it, another special set of vertices which we call the reservoir. On the remaining graph, we shall find an almost perfect covering of its vertices by~$i_0$ disjoint~$(\vec x_i, \vec y_i; k)$-paths with~$\vec x_i$,~$\vec y_i$ being~$\zetac$-connectable.
	
	 The reservoir~$R \subseteq V$ will be used to connect the absorbing path and the paths in the almost perfect covering to attain an almost spanning cycle. 
	 For that it is convenient to choose the set~$R$ in such a way that for any connectable 
	 $\vec x$ and $\vec y$  there are still ``many'' $(\vec x, \vec y; k)$-paths having all their inner vertices in~$R$.
	 In order to define the reservoir, consider the induced subgraph 
	\[G' = G[(V \setminus V(P_A) ) \cup V(\vec x_A) \cup V(\vec y_A)] \,.\]
	Since the number of deleted vertices is~$|V(P_A)| - 2k \leq \zetaa n/2 \leq \mu n/4$, Lemma~\ref{lem:muprime} 
	shows that~$G'$ is $\mu/2$-inseparable and, since $\mu n/4<n/2$, it follows from the $(\rho,d)$-denseness of $G$ 
	that~$G'$ is also~$(4 \rho, d)$-dense. Moreover, if~$\vec x$ is a clique in~$G'$ that is~$\zetaa$-connectable in~$G$, then it is~$\zetaa/2$-connectable in~$G'$. By our choice of constants in~\eqref{eq:mainZeta} and~\eqref{eq:mainRho}, it follows from the Connecting Lemma that if~$\vec x$ and~$\vec y$ are two disjoint $\zetac$-connectable cliques in~$G'$, then there are at least~$\xic(n/2)^m$ distinct $(\vec x, \vec y;k)$-paths with~$m$ inner vertices in~$G'$, for some~$m = m(\vec x, \vec y) \leq \mc$.
	
	We wish the reservoir to be disjoint from~$P_A$, thus consider the induced subgraph $G - V(P_A) = G' - (V(\vec x_A) \cup V(\vec y_A))$. The number of $(\vec x, \vec y;k)$-paths in~$G'$ intersecting~$\vec x_A$ or~$\vec y_A$ is at most~$2k \cdot m \cdot (n/2)^{m-1}$, implying that for~$n$ sufficiently large, there are at least
	\begin{equation}\label{eq:mainPaths}
		\xic \Big( \frac{n}{2} \Big)^m - 2k \cdot m \cdot \Big( \frac{n}{2} \Big)^{m-1} \geq \frac{\xic}{2}\Big( \frac{n}{2} \Big)^{m(\vec x, \vec y)}\,,
	\end{equation}
	$(\vec x, \vec y;k)$-paths with inner vertices in~$G - V(P_A)$. Note that this is true for any two disjoint $\zetac$-connectable cliques in~$G'$, which in particular allows~$\vec x_A$ or~$\vec y_A$ to be one of them. 
	
	For the reservoir~$R$ we choose vertices from~$V\setminus V(P_A)$ independently with probability
	\begin{equation}
		p = \frac{\alphaa}{4}\,.
	\end{equation}
	The desired properties for the reservoir are  
	\begin{enumerate}[label=(\roman*)]
		\item $|R| \leq \alphaa n /2$ and \label{prop:mainR}
		\item for any two disjoint~$\zetac$-connectable cliques~$\vec x$ and~$\vec y$ in~$G'$, there are at least
		\[\frac{\xic}{4} \Big(\frac{\alphaa}{4}\Big)^m \Big(\frac{n}{2}\Big)^m \,,  \]
		$(\vec x, \vec y;k)$-paths with all their~$m = m(\vec x, \vec y)$ inner vertices in~$R$. \label{prop:mainR2}
	\end{enumerate}
	For property~\ref{prop:mainR} we observe that Markov's inequality implies
	\[ \PP \Big( |R| \geq  \frac{\alphaa n}{2} \Big) \leq \frac{ 2 \EE |R|} { \alphaa n} =  \frac{n - |V(P_A)|} { 2 n} < \frac{1}{2}\,.\]
	For property~\ref{prop:mainR2}, let~$\vec x$ and $\vec y$ be two disjoint~$\zetac$-connectable cliques in~$G'$ and define~$X$ to be the random variable that counts the number of $(\vec x, \vec y;k)$-paths with all their~$m = m(\vec x, \vec y)$ inner vertices in~$R$. Note that the inclusion or exclusion of a vertex in~$R$ affects~$X$ by at most~$m\cdot n^{m - 1}$ and that
	\[\EE X \overset{\eqref{eq:mainPaths}}{\geq} \frac{\xic}{2} \Big( \frac{n}{2} \Big)^m \Big( \frac{\alphaa}{4} \Big)^m\,. \]
	Consequently, the Azuma-Hoeffding inequality (see, e.g., \cite{JLR}*{Corollary~2.27}) asserts that 
	\begin{align}\label{eq:mainAzuma}
	\PP  \Big( X \leq \frac{\xic}{4} \Big(\frac{\alphaa}{4}\Big)^m \Big(\frac{n}{2}\Big)^m \Big) &
	 \leq \exp \Big( - \frac{ ( \xic /4)^2 ( \alphaa / 4)^{2m} ( n / 2)^{2m}  } {2 n \cdot m^2 n^{2m - 2}} \Big)  \nonumber\\
	& \leq \exp \Big(  -  \frac{\xic^2}{32}  \cdot \frac{ \alphaa^{2\mc} }{8^{2\mc} \mc^2 } \cdot n \Big)\,.
	\end{align}
	Since there are at most~$n^{2k}$ pairs~$(\vec x,\vec y)$, the union bound and~\eqref{eq:mainAzuma} show that the probability of~$R$ not having property~\ref{prop:mainR2} is less than~$1/2$ for sufficiently large~$n$. Hence, there is some set~$R$ satisfying both~\ref{prop:mainR} and~\ref{prop:mainR2}. 
	
	Set aside~$P_A$ and~$R$ and cover the vertices of~$G - (V(P_A) \cup V(R))$ by disjoint $k$-paths until at most~$\alphaa n /2$ vertices are left uncovered. 
	We obtain these~$k$-paths by repeated applications of the Path Lemma. 
	Assume we already have~$\{P_1, \dots, P_{i}\}$, where for~$j = 1, \dots, i$ the $k$-path~$P_j$ is a~$(\vec x_j, \vec y_j; k)$-path with~$\vec x_j$,~$\vec y_j$ being~$\zetac$-connectable in~$G'$ and
	\begin{equation}\label{eq:mainSizePj}
			|V(P_j)| \geq \frac{\alphaa \zetap }{2} n.
	\end{equation}
	Let~$L = V - V(P_A) - V(R) - \bigcup_{j = 1}^i V(P_j)$ be the set of vertices not yet covered and suppose that~$|L| \geq \alphaa n /2$. 
	Hence, the subgraph~$G[L]$ is $(4\rho/ \alphaa^2, d)$-dense. By the choice in~\eqref{eq:mainRho} and the Path Lemma, there is a~$(\vec x_{i+1}, \vec y_{i+1}; k)$-path~$P_{i+1}$ in~$G[L]$ with
	\[|V(P_{i+1})| \geq \zetap |L| \geq \frac{\alphaa \zetap }{2} n\,,\]
	and~$\vec x_{i+1}$,~$\vec y_{i+1}$ being $\zetap$-connectable in~$G[L]$. 
	By the choice in~\eqref{eq:mainZeta}, we have
	\[\zetap |L| \geq \zetac n \geq \zetac |V(G')|\,,\]
	and hence,~$\vec x_{i+1}$,~$\vec y_{i+1}$ are $\zetac$-connectable in~$G'$.
	Therefore, we may enlarge the partial~$k$-path covering until the set of leftover vertices has size at most
	\begin{equation}\label{eq:mainL}
	|L| < \frac{\alphaa}{2} n\,.
	\end{equation}
	Let~$\{P_1, \dots, P_{i_0} \}$ be such a family of~$k$-paths. Note that~\eqref{eq:mainSizePj} yields
	\begin{equation}\label{eq:maini0}
	 i_0 < \frac{n}{\alphaa \zetap/2 \cdot n} = \frac{2}{\alphaa \zetap}\,.
	\end{equation}
	
	The next step is to connect the $k$-paths in~$\{ P_A = P_0, P_1, \dots, P_{i_0} \}$ using the reserved vertices in~$R$ to obtain the $k$-th power of an almost spanning cycle.
	Assume we already obtained for some~$0 \leq j \leq i_0$ a~$(\vec x_A, \vec y_j;k)$-path~$Q^j$ containing the paths~$P_A, P_1, \dots, P_j$ and such that
	\begin{equation}\label{eq:mainUsedR}
	|V(Q^j) \setminus (V(P_A) \cup \dots \cup V(P_{j}))| = |V(Q^j) \cap R| \leq j \cdot \mc\,.
	\end{equation}
	 If~$j < i_0$, we will connect~$Q^j$ to $P_{j+1}$ and obtain the~$(\vec x_A, \vec y_{j+1};k)$-path~$Q^{j+1}$ that in addition contains~$P_{j + 1}$ and satisfies~$|V(Q^{j+1}) \cap R| \leq (j + 1) \cdot \mc$.
	 For~$j = i_0$, we have that~$Q^{i_0}$ is a~$(\vec x_A, \vec y_{i_0};k)$-path including all~$\{ P_A, P_1, \dots, P_{i_0} \}$ and we will connect~$\vec y_{i_0}$ to~$\vec x_{i_0 + 1} = \vec x_A$, obtaining the~$k$-th power of a cycle.
	 
	 By property~\ref{prop:mainR2} of the reservoir, since~$\vec y_j$ and~$\vec x_{j+1}$ are~$\zetac$-connectable cliques in~$G'$, there are at least
	 \[
	 \frac{\xic}{4} \Big(\frac{\alphaa}{4}\Big)^m \Big(\frac{n}{2}\Big)^m \,,
	 \]
	 $(\vec y_{j}, \vec x_{j+1};k)$-paths with all its~$m = m(\vec y_{j}, \vec x_{j+1})$ inner vertices in~$R$.
	 We need to ensure that one of these $k$-paths is disjoint from~$Q^j$. 
	 By~\eqref{eq:mainUsedR}, the number of~$k$-paths with~$m$ inner vertices that intersect~$Q^j$ is at most
	 \[
	 j \cdot \mc \cdot m \cdot n^{m-1} \overset{\eqref{eq:maini0}}{<} \frac{2}{\alphaa \zetap} \cdot \mc^2 \cdot n^{m-1} \overset{\eqref{eq:mainN}}{\leq} \frac{\xic}{4} \Big(\frac{\alphaa}{8}\Big)^{\mc} n^m \,.
	 \] 
	 Hence, there is a $(\vec y_{j}, \vec x_{j+1};k)$-path with all its~$m \leq \mc$ inner vertices in~$R$ that is disjoint from~$Q^i$. If~$j < i_0$, this~$(\vec y_{j}, \vec x_{j+1};k)$-path can be used to build~$Q^{j+1}$, which satisfies that
	 \[|V(Q^{j+1}) \cap R| \leq j \cdot \mc + \mc = (j+1)\cdot \mc\,.\]
	 If~$j = i_0$, use this~$(\vec y_{i_0}, \vec x_{i_0+1};k)$-path to close the~$k$-th power of an almost spanning cycle~$H'$. 
	 
	 The vertices from~$G$ which are not in~$H'$ are those from~$L$, which were not covered by the almost perfect~$k$-path covering, plus the vertices in~$R$ that were not used to connect the paths in~$\{ P_A, P_1, \dots, P_{i_0} \}$. Hence, 
	 \[
	 |V \setminus V(H')| \leq |L| + |R| \overset{ \eqref{eq:mainL}}{\leq} \frac{\alphaa}{2} n + \frac{\alphaa}{2} n\,.
	 \]
	 Since~$P_A$ is a segment of~$H'$ and~$P_A$ is~$\alphaa$-absorbing, we may replace~$P_A$ by  a~$(\vec x_A, \vec y_A; k)$-path with vertex set~$V(P_A) \cup (V \setminus V(H'))$ and obtain the desired~$k$-th power of a Hamiltonian cycle in~$G$.
\end{proof}

\section{Embedding graphs of small bandwidth}
\label{sec:bandwidth}
In this section we sketch the proof of Theorem~\ref{thm:main2}, which is based on the regularity 
method for graphs and on Theorem~\ref{thm:main}. There are quite a few examples of results 
obtained by the regularity method, which are based on reductions to simpler or seemingly weaker results that 
are applied to the reduced graph of the regular partition obtained by an application of 
Szemer\'edi's regularity lemma~\cite{Sz78} (see, e.g.,~\cite{KS96}*{Sections~2, 4-6} and~\cite{Ko99}).
In particular, the proofs of the bandwidth theorems in~\cites{BST09,staden2018} were based on reductions to the corresponding theorems for 
powers of Hamiltonian cycles. We will follow the same route and start the discussion by recalling this approach.
For the discussion below we assume the reader to be familiar with the regularity method for graphs and the \emph{blow-up lemma}
from~\cite{blowup97}.

\subsection{Sketch of the proof of Theorem~\ref{thm:main2}}
Given a bounded degree graph $H=(V_H,E_H)$ of small linear bandwidth 
and a uniformly dense and inseparable graph $G=(V_G,E_G)$, we 
apply Szemer\'edi's regularity lemma~\cite{Sz78} and obtain a regular partition $V_0\dcup V_1\dcup\dots\dcup V_t=V_G$ and a reduced graph~$R=([t],E_R)$.
Moreover, by a careful choice of the involved constants we may ensure that $R$ satisfies essentially the same 
assumptions as~$G$, i.e., we ensure that~$R$ itself is again 
uniformly dense and inseparable. Consequently, we can appeal to Theorem~\ref{thm:main} and
obtain a $k$-th power of a Hamiltonian cycle~$C$ in the reduced graph~$R$ for any fixed (not yet specified) $k\in\NN$. 
Moreover, without loss of generality we may assume that the (cyclic) ordering of~$V(C)$ matches the one from $[t]$ and that $k+1$ divides $t$.

Next we prepare for an application of the blow-up lemma of Koml\'os, S\'ark\"ozy, and Szemer\'edi~\cite{blowup97}.
For that we fix a $K_{k+1}$-factor $F\subseteq C$ and for the price of moving some vertices from every $V_i$ for $i\in[t]$ to 
the exceptional class $V_0$ of the regular partition, we can ensure that those pairs covered by cliques from~$F$ become super-regular.
Let $V'_0\dcup V'_1\dcup\dots\dcup V'_t=V_G$ be the regular partition after this step.
In the next step we want to:
\begin{enumerate}[label={($\star$)}]
\item\label{it:A} redistribute the vertices of~$V'_0$ to the other vertex classes without compromising the 
	super-regularity of the pairs covered by~$F$.
\end{enumerate}
Resolving~\ref{it:A} presents the main challenge in the proof of Theorem~\ref{thm:main2} and we 
discuss this in Section~\ref{sec:except}. For now let $V''_1\dcup\dots\dcup V''_t=V_G$ be the resulting vertex partition of $G$ after the redistribution.

After preparing the graph $G$ we want to find a ``matching'' 
partition $W_1\dcup\dots\dcup W_t=V_H$ of the vertex set of $H$ into independent sets.
In particular, we would like to achieve that~$|W_i|=|V''_i|$ for every $i\in[t]$ and that applying the blow-up lemma for every 
clique in~$F$ yields a partial embedding of pieces of $H$ induced on the corresponding vertex classes into~$G$. However, in addition 
we have to make sure that these partial embeddings lead to a full embedding $H\hookrightarrow G$. For that we have to ensure 
that the necessary adjacencies between the pieces of the partial embeddings are preserved. 
Owing to the small bandwidth of~$H$ only few vertices of every piece may have neighbours in the next piece, 
where the ordering of the pieces reflects the bandwidth ordering of $V_H$. This allows us to embed these vertices first 
(using the greedy embedding strategy from~\cite{CRST83}) and then utilise the so-called \emph{image restrictions} from the blow-up lemma.
For the greedy embedding of those connections, we need that the involved pairs are covered by~$C$. 

Summarising the discussion above, we
want to find a partition $W_1\dcup\dots\dcup W_t$ of $V_H$ satisfying
\begin{enumerate}[label=\rmlabel]
\item\label{it:i} $W_i$ is independent in $H$ for every $i\in [t]$,
\item\label{it:ii} $|W_i|=|V''_i|$ for every $i\in [t]$, and
\item\label{it:iii} $\phi\colon V_H\to [t]$, defined by $\phi^{-1}(i)=W_i$, is a graph homomorphism from~$H$ to~$C$.
\end{enumerate}
The maximum degree restriction $\Delta(H)\leq \Delta$ implies $\chi(H)\leq \Delta+1$. Moreover, the Hajnal-Szemer\'edi theorem~\cite{hajnal1970}
yields for every $U\subseteq V_H$ an equitable $(\Delta+1)$-colouring of~$H[U]$.
Hence, we may simply split $V_H$ into $t/(\Delta+1)$ 
intervals (according to the bandwidth order) of roughly the same size and then fix an equitable $(\Delta+1)$-colouring of each interval. This yields 
the desired partition $W_1\dcup\dots\dcup W_t$ of $V_H$, which obviously satisfies~\ref{it:i}. 

Fixing $k\geq \Delta$ implies that $C$ contains a $K_{\Delta+1}$-factor $F$ and this 
may allow us to apply the blow-up lemma for the partial embeddings of these intervals. For that we also 
need to address~\ref{it:ii}. In the process of preparing the graph $G$, we may have moved a few vertices leading to slightly 
different sizes of the vertex classes $V''_1,\dots,V''_t$. However, using the super-regularity of the pairs covered by cliques 
from~$F$ and the uniform density of~$G$ (applied to $G[V''_i]$) tells us that we can move a few vertices between classes covered by the same 
clique~$K$ from~$F$ to balance those classes to restore~\ref{it:ii}, without affecting the super-regularity of the pairs covered by~$K$.

For part~\ref{it:iii} it will be convenient that all pairs between two consecutive $K_{\Delta+1}$'s from~$F$ are regular.
In fact, this would even allow the small alterations for balancing the classes corresponding to one 
$K_{\Delta+1}$ mentioned above, without compromising the regularity. We can achieve this by simply choosing $k=2\Delta+1$ and having 
$C$ to be a $(2\Delta+1)$-st power of a Hamiltonian cycle in~$R$ and fixing a $K_{\Delta+1}$-factor $F\subseteq C$.

This concludes the outline of the proof of Theorem~\ref{thm:main2}. It remains to address~\ref{it:A} (see Section~\ref{sec:except}), while the other steps 
of the proof can be obtained by adaptations of the proofs from~\cites{BST09,staden2018} and the details can be found 
in~\cite{Ebsen}.

\subsection{Redistributing the vertices of the exceptional class}
\label{sec:except}
For addressing the problem raised in~\ref{it:A} it will be convenient to consider the following strengthening of Theorem~\ref{thm:main}.
\begin{thm}\label{thm:main1b}
	For every~$d$,~$\mu \in  (0,1]$, and~$k \in \NN$ there exist~$\rho$, $\alpha$, $\gamma > 0$ and~$t_0$ such that the following holds.
	Suppose $R=([t],E_R)$ is a $(\rho, d)$-dense and $\mu$-inseparable graph on~$t \geq t_0$ vertices 
	and subsets $U_1,\dots,U_m\subseteq [t]$ each of size at least $\mu t$ for some $m\leq 2^{\gamma t}$ are given.
	
	Then~$R$ contains the~$k$-th power of a Hamiltonian cycle $C$ with the additional property that for every $i\in [m]$
	there are at least $\alpha t$ cliques $K_{k}$ contained in~$C[U_i]$.
\end{thm}
\begin{proof}[Proof of Theorem~\ref{thm:main1b} (Sketch)]
	There is only one difference between Theorems~\ref{thm:main1b} and~\ref{thm:main}.
	It concerns the additionally given vertex subsets $U_1,\dots,U_m\subseteq V=V(R)$ of size~$\Omega(|V|)$ for which we 
	require that the guaranteed~$k$-th power of a Hamiltonian cycle shares $\Omega(|V|)$ many~$K_{k}$ with each of these sets.
	
	This additional restriction can be achieved by adjusting the  proof of the Absorbing Path Lemma (Lemma~\ref{lem:abspath}).
	Since each of the sets $U_i$ has linear size, Lemma~\ref{lem:dense} yields~$\Omega(|V|^{k+1})$ cliques 
	$K_{k+1}$ in $U_i$. Consequently, by a standard averaging argument at least $\Omega(|V|^{k})$ $K_k$'s
	are $\zeta$-connectable for some sufficiently small $\zeta=\zeta(d,\mu)>0$. Hence, following the proof of the 
	Absorbing Path Lemma (Lemma~\ref{lem:abspath}) we may consider a random collection of $K_{k}$'s of size $\Omega(|V|)$ and need to ensure
	that with positive probability it contains~$\Omega(|V|)$ such $K_{k}$'s for every given set $U_i$. For a fixed set~$U_i$ this holds with 
	probability at least~$1-\exp(-\Omega(|V|))$ and by the union bound this can be achieved for all sets $U_1,\dots,U_m$ simultaneously 
	as long as $m=2^{o(|V|)}$, which is enforced by the assumption of Theorem~\ref{thm:main1b}. The rest of the proof of Lemma~\ref{lem:abspath}
	would be analogous. As a result the absorbing path $P_A$ is not only $\alpha$-absorbing, but also induces
	$\Omega(|V|)$ cliques~$K_{k}$ within every given set $U_1,\dots,U_m$. Finally, we remark that none of these $K_{k}$ would be 
	altered during the absorption process, since they can be chosen disjoint from the $K_{2k}$ used for the absorption property. 
	Therefore, the guaranteed~$k$-th power of a Hamiltonian cycle~$C$ has the desired property.
\end{proof}

Below we explain how Theorem~\ref{thm:main1b} can be used to address~\ref{it:A} in the proof of Theorem~\ref{thm:main2}.
In the proof of Theorem~\ref{thm:main2} we want to find a $(2\Delta+1)$-st power of a Hamiltonian cycle~$C$ in the reduced 
graph~$R$ and we fix a $K_{\Delta+1}$-factor $F\subseteq C$, where each clique consists of  consecutive vertices from~$C$. 
Note that we could easily
redistribute the vertices of $V'_0$, if we would know that for every vertex $v\in V_0$ there exists a clique $K$ in $F$ such that  
$|N_G(v)\cap V'_j|=\Omega(|V'_j|)$ for all $j\in V(K)$. In fact, in such a case we could simply add~$v$ to any of the classes $V'_j$
corresponding to a vertex of~$K$ without affecting the super-regularity of the pairs covered by~$K$. Moreover, if  for every 
vertex $v\in V'_0$ there would be at least 
$\alpha t$ such cliques $K$, then we could redistribute all vertices of $V'_0$
in a fairly balanced way, i.e., by enlarging the sizes of the classes $V'_i$ for every $i\in [t]$ 
by at most $\frac{|V'_0|}{\alpha t}$ vertices and this would resolve the problem raised in~\ref{it:A}.

For that we would like to apply Theorem~\ref{thm:main1b} to the reduced graph~$R$ with additional sets 
\[
	U_v=\big\{j\in[t]\colon |N_G(v)\cap V_j|=\Omega(|V_j|)\big\}\,.
\]
On the one hand, the $\mu$-inseparability of $G$ would imply that $|U_v|\geq \mu' t$ for some $\mu'>0$. On the other hand,
since the vertex set $V'_0$ depends on $F\subseteq C$ and we may move some vertices to $V_0$ to establish the super-regularity of the pairs covered by~$F$, we would need to consider~$U_v$ for every $v\in V_G$ and such a na\"\i ve approach would require 
$m=n\gg2^{\gamma t}$.

We may consider vertices $v$, $w\in V_G$ to be equivalent ($v\sim w$),
if $U_v=U_w$, which allows us to reduce $m$ to $2^{\mu'' t}$ for some
$\mu''>0$ depending on $\mu$. However, since the parameter $\gamma>0$ in Theorem~\ref{thm:main1b} depends on $\mu$, 
such a basic improvement
on the number of sets would not suffice for an application of Theorem~\ref{thm:main1b}. 

We resolve this issue by considering a random refinement of $V_1\dcup\dots\dcup V_t$, where each class~$V_i$ is split randomly into $s$ sets 
$V_{i,1}\dcup\dots\dcup V_{i,s}=V_i$. Let $S$ be the corresponding reduced graph of the refined partition.
The local characterisation of regularity through the number of four-cycles combined with the Azuma-Hoeffding inequality 
tells us that every subpair $(V_{i,a},V_{j,b})$ will inherit the regularity from $(V_i,V_j)$ as long as $s=o((n/\log n)^{1/8})$.
Hence, we may consider such a random refinement for some $s\gg1/\gamma$ and the resulting reduced graph $S$ would be simply an $s$-blow-up 
of~$R$. Consequently, the reduced graph~$S$ is also inseparable and uniformly dense.

Moreover, the randomness can be further utilised to show that the 
equivalence relation~$\sim$ would not be affected. More precisely, Chernoff's inequality implies that if $U_v=U_w$ for the reduced graph $R$, then $U^S_v=U_w^S$ for the sets 
\[
	U^S_v=\big\{(j,\sigma)\in[t]\times [s]\colon |N_G(v)\cap V_{j,\sigma}|=\Omega(|V_{j,\sigma}|)\big\}\,.
\]
Consequently, we can apply Theorem~\ref{thm:main1b} (with $k=2\Delta+1$) for the reduced graph~$S$ on $st$ vertices 
with 
\[
	m\leq 2^{\mu'' t}\leq 2^{\gamma st}\,.
\] 
As a result, for every vertex $v\in V$ there are at least 
$\alpha st$ $K_{2\Delta +1}$ contained in~$C[U^S_v]$ and, hence, $F[U^S_v]$
contains at least $\alpha st$ different $K_{\Delta+1}$. This way we can redistribute the vertices of $V'_0$ as discussed above 
and this addresses~\ref{it:A} in the proof of Theorem~\ref{thm:main2}.

This concludes the discussion of the proof of Theorem~\ref{thm:main2}. For the technical details of this approach we refer 
to~\cite{Ebsen}.

\section{Concluding remarks}
\label{sec:remarks}

	Finding a common generalisation of the approximate version of Theorem~\ref{thm:KSSz}
	and of Theorem~\ref{thm:main} is an interesting open problem. In other words, we want to 
	weaken the assumptions of Theorem~\ref{thm:main} in such a way that, 
	on the one hand, for every $\eps>0$, every sufficiently large $n$-vertex graph with minimum degree 
	$(\frac{k}{k+1}+\eps)n$ would satisfy them and, on the other hand,
	they still ensure the existence of a $k$-th power of a Hamiltonian cycle. 
	
In fact, for $k = 1$, 
$\mu$-inseparability alone yields an appropriate Connecting Lemma (see Lemma~\ref{lem:muinsep}).
Moreover,  a combination of some ideas from~\cite{PR} and~\cite{R3S2} shows that it supplies  
sufficiently strong absorbers for such an approach. 
However, considering $n$-vertex graphs that are the complement of a clique on $(1-\mu)n$ vertices, shows that $\mu$-inseparability alone does not suffice for the existence of Hamiltonian cycles.

The missing part for this problem is an appropriate relaxation of the $(\rho,d)$-denseness assumption that still guarantees an almost perfect path cover. This is rendered by the following property, which is closely related to the notion of \emph{robust (out)expanders}
utilised by K\"uhn, Osthus and their collaborators in their work on Hamiltonian cycles (and decompositions thereof) in (directed) graphs
(see, e.g.,~\cite{KO-icm}*{Section 3.4} and the references therein).

\begin{dfn}\label{def:robustMatchable}
For $\rho > 0$, $d\in[0,1]$   a graph $G=(V,E)$ is \emph{$(\rho, d)$-robustly matchable}, if for all subsets $U \subseteq V$ we have
\[
e(U) \geq d\frac{|U|^{2}}{2} - \rho n^{2},
\]
or
\[
 |U| \leq \frac{n}{2} + \rho n \text{\quad and \quad} \big| \big\{ v \in V \setminus U\colon |N(v) \cap U| \geq d|U| - \rho n \big\} \big| \geq |U| - \rho n\,.
\]
\end{dfn}
Obviously every $(\rho, d)$-dense graph is $(\rho, d)$-robustly matchable and 
it is not hard to show that for $\rho\ll d$, such graphs contain matchings that cover almost all vertices.
In fact, the same can be inferred for the reduced graph after a suitable application of Szemer\'edi's regularity lemma.
The regular pairs in the reduced graph covered by such a matching contain a collection of relatively few disjoint paths that cover 
the vertices of the regular pair. In other words, one can show that $(\rho, d)$-robustly matchable graphs admit an 
almost perfect path cover that can be used for an absorption based proof.

The discussion above outlines an approach based on the absorption method that shows that
for every $d>0$, there exist $\rho$ and $\mu>0$ such that every sufficiently large 
$(\rho, d)$-robustly matchable and $\mu$-inseparable graph is Hamiltonian. The details of such a proof 
appear in~\cite{Ebsen}. In this way, we obtain strengthening of Theorem~\ref{thm:main} for $k=1$. 
Moreover, it is straightforward to check that for every $\eps>0$, every sufficiently large graph~$G=(V,E)$ 
with $\delta(G)\geq (\frac{1}{2} +\eps)|V|$ is $(\rho, d)$-robustly matchable and $\mu$-inseparable for appropriate 
parameters~$d$, $\rho$, and~$\mu$. In other words, 
this yields a common generalisation of the approximate version of Dirac's theorem (Theorem~\ref{thm:KSSz} for $k=1$)
and Theorem~\ref{thm:main} for~$k=1$. 

For powers of Hamiltonian cycles, 
we remark that $\mu$-inseparability does not suffice to ensure a Connecting Lemma nor the 
existence of appropriate absorbers. For example, random bipartite graphs of 
edge density $2\mu+o(1)$ are $\mu$-inseparable, while a Connecting Lemma and absorbers 
for $k$-th powers of Hamiltonian cycles must contain cliques of order~$k+1$. 
However, a recursive definition of inseparability within vertex neighbourhoods seems to be 
sufficient in this context. Combining this with a notion that robustly ensures 
almost perfect $K_{k+1}$-factor's might lead to a common generalisation of the approximate version of Theorem~\ref{thm:KSSz} 
and Theorem~\ref{thm:main} for arbitrary $k\geq 1$ and we intend to come back to this in the near future.

\end{document}